\RequirePackage{fix-cm}

\documentclass[smallcondensed]{svjour3}     
\smartqed  

\usepackage{graphicx}

\usepackage{amsmath,amssymb,amsbsy}

 \newcommand{\dis}{\displaystyle}

\newtheorem{thm}{\textsc{Theorem}}
\newtheorem{pro}{\textsc{Proposition}}
\newtheorem{cor}{\textsc{Corollary}}
\newtheorem{lem}{\textsc{Lemma}}
\newtheorem{rmq}{\textsc{Remark}}

\newtheorem{exemple}{\textsc{Example}}

\begin{document}

\title{{\L}S condition for filled Julia sets in $\mathbb{C}$}

\author{Fr\'ed\'eric PROTIN
}

\institute{INSA de Toulouse, Institut Math\'ematique de Toulouse \at
              135 avenue de Rangueil, 31400 Toulouse, France \\
              Tel.: +33-684729787\\
              \email{fredprotin@yahoo.fr}      }


\maketitle

\begin{abstract}
In this article, we derive an inequality of {\L}ojasiewicz-Siciak type for certain sets arising in the context of the complex dynamics in dimension 1. More precisely, if we denote by $dist$ the euclidian distance in $\mathbb{C}$, we show that the Green function $G_K$ of the filled Julia set $K$ of a polynomial such that $\mathring{K}\neq \emptyset$ satisfies the so-called {\L}S condition $\displaystyle G_A\geq c\cdot dist(\cdot, K)^{c'}$ in a neighborhood of $K$, for some constants $c,c'>0$. Relatively few examples of compact sets satisfying the {\L}S condition are known. Our result highlights an interesting class of compact sets fulfilling this condition. The fact that filled Julia sets satisfy the {\L}S condition may seem surprising, since they are in general very irregular. In order to prove our main result, we define and study the set of obstruction points to the {\L}S condition. We also prove, in dimension $n\geq 1$, that for a polynomially convex and L-regular compact set of non empty interior, these obstruction points are rare, in a sense which will be specified.
\end{abstract}

\section{Introduction}
\label{intro}

We call {\it{pluricomplex Green function}} $G_A$ of a compact set $A\subset\mathbb{C}^n$, $n\geq 1$, the plurisubharmonic function defined as$$\displaystyle G_{A}:= {\sup}^* \bigg\{ v \in PSH(\mathbb{C}^n):\text{ } v|_{A}\leq 0,\text{ } v(z)\leq \frac{1}{2}\log(1+\|z\|^2)+O(1)\bigg\},$$where ${\sup}^*$ denotes the upper semi-continuous regularization of the upper envelope, and $PSH(\mathbb{C}^n)$ denotes the set of plurisubharmonic functions in $\mathbb{C}^n$. The set $A$ is called {\it{L-regular}} if $G_A$ is continuous. In this case, the set $\{G_A = 0\}$ is the polynomially convex envelope $\hat{A}$ of $A$. We also consider, for an open bounded set $U\subset \mathbb{C}^n$, the {\it{Green function of $A\subset U$ relative to $U$}} defined by $$\displaystyle G_{A,U}:= {\sup}^* \big\{ v \in PSH(U):\text{ } v\leq 0,\text{ } v|_{A}\leq -1\big\}.$$

Let $U_a:=\{G_A<a\}$ for $a\in\mathbb{R}^{+}\setminus \{0\}$. If $A$ is not pluripolar and $\hat{A}\subset U_a$, then a relation between $G_A$ and $G_{A,U_a}$ holding in $U_a$ is given by Proposition 5.3.3 in \cite{K}: \begin{equation}\label{relation}
\dis G_A=a(G_{A,U_a}+1).
\end{equation}

A compact $A\subset \mathbb{C}^n$ is said to {\it{satisfy the {\L}S} condition} if there exists an open set $U$ containing it and two constants $c,c'>0$ such that its pluricomplex Green function $G_A$ verifies the following regularity condition : $$\forall z\in U, G_A(z)\geq c\cdot dist(z, A)^{c'},$$where $dist$ denotes the euclidean distance (see for instance \cite{BG} or \cite{BG2}).\\
\noindent For technical reasons, we will take in this work $c' = \frac{1}{c}$, which does not change the definition, as we are not interested here in finding the optimal constants.\\
 
On a compact set $A\subset \mathbb{C}^n$ verifying the {\L}S condition, as well as the HCP condition (i.e. the H\"olderian continuity of $G_A$), for example a semi-algebraic compact set, we have the rapid approximation property of continuous functions by polynomials. Relatively few examples of compacts satisfying the {\L}S condition are known. Some examples are given in \cite{PiI}. Let us also note that Pierzcha{\l}a showed in \cite{PiII} that a compact verifying the {\L}S condition is polynomially convex. Bia{\l}as and Kosek \cite{BK} construct such sets using holomorphic dynamics.

In the same vein, we show that the so-called filled Julia sets in $\mathbb{C}$ satisfy the {\L}S condition. More precisely, our main goal is to show the following result concerning the filled Julia set of a polynomial $f:\mathbb{C}\rightarrow \mathbb{C}$, i.e. the set of points $z\in\mathbb{C}$ whose orbit $(f^n(z))_n$ is bounded : \\
 
\noindent {\bf{Theorem. }}{\it{The filled Julia set of a polynomial $f:\mathbb{C}\rightarrow \mathbb{C}$ of degree $\geq 2$, if its interior is non empty, satisfies the {\L}S condition.}}\\

The differentials operators operators $\partial$ and $\overline{\partial}$ will be understood in the sense of currents. Recall that a continuous function $u$ from an open set of $\mathbb{C}^n$ into $\mathbb{R}$ is pluriharmonic (harmonic if $n=1$) if and only if $\partial \overline{\partial} u =0$ (see for example Theorem 2.28 in \cite{L}). \\

In Section 2, we recall some definitions and elementary facts about holomorphic dynamics in one dimension, and we prove a useful lemma concerning the regularity of filled Julia sets. More precisely, we prove that the filled Julia set $K$ of a polynomial of degree $d\geq 2$ with non-empty interior satisfies $\overline{\mathring{K}}=K$. In Section 3, we define in $\mathbb{C}^n$, $n\geq 1$, the set of obstruction points to the {\L}S condition, and we prove that the complementary of this set is big, in a sense which will be specified. Section 4 is devoted to the proof of the main theorem previously stated.

\section{Dynamics in $\mathbb{C}$}

We start by recalling some definitions related to one-dimensional holomorphic dynamics. Let us consider a polynomial $f:\mathbb{C}\rightarrow \mathbb{C}$ of degree $d\geq 2$. 

We call \textit{Fatou set} of $f$, denoted $\mathcal{F}$, the largest open subset in which the family of iterations $f^n$ is equicontinuous. 

\textit{The Julia set of} $f$, denoted $J$, is the complement of $\mathcal{F}$ in $\mathbb{C}$. Let us note for what follows that $J$ is not pluripolar. 

We call \textit{filled Julia set of} $f$ the set $K$ of points $z\in \mathbb{C}$ whose orbit $(f^n(z))_n$ is bounded.
Note that $K$ is compact, as $\infty$ is a superattractive fixed point of $f$, hence belonging to $\mathcal{F}$. The complement of $K$ is the basin of attraction of infinity. We have $\partial K = J$ and $G_K = G_J$. \\
Under very little restrictive conditions, the set $K$ is of non-empty interior. It is the case, for instance, for $f(z) = z^2+a$, with $a$ in the interior of the Mandelbrot set.\\

We construct the subharmonic function $G:\mathbb{C}\rightarrow\mathbb{R}^+$, limit in $L^1_{loc}$ of the sequence $ (\log(1+|f^n|)/d^n)_n$. It is known that $G$ is continuous (and even H\"olderian), harmonic in $\mathcal{F}$, and that it verifies $G(z)=0$ if and only if $z\in K$, and also that $G(z)-\log|z| = O(1)$ at infinity. By uniqueness, $G$ is therefore the pluricomplex Green function of $K$ (and of $J$). It satisfies by construction the invariance property \begin{equation}\label{invariance}G\circ f = d\cdot G.\end{equation}The measure $\frac{i}{\pi}\partial \overline{\partial}  G$ is a probability measure of support exactly $J$ (see e.g. \cite{G}). We first show a preliminary lemma about the filled Julia set.

\begin{lem}\label{julia}
The filled Julia set $K$ of a polynomial of degree $d\geq 2$ with non-empty interior satisfies $\overline{\mathring{K}}=K$.
\end{lem}

\begin{proof}
First recall the following equidistribution result (see e.g. Theorem 1.10 in \cite{G} or Theorem 6.1 in \cite{FS2}). Let $G$ the Green function associated to a polynomial $f$ of degree $d\geq 2$. Then, for all $x\in \mathbb{C}$ (except possibly for a totally invariant set consisting of at most two points in the Fatou set, see e.g. \cite{Be}), we have the following weak convergence of measures :\begin{equation}\label{equidistribution}\displaystyle\lim_{n\rightarrow +\infty}\frac{1}{d^n}f^{n*}\delta_x:=\lim_{n\rightarrow +\infty}\frac{1}{d^n}\sum_{y:f^n(y)=x}\delta_y = \frac{i}{\pi}\partial \overline{\partial} G,\end{equation}

\noindent where $\delta_x$ is a Dirac measure with support in $\mathring{K}$, except for the exceptional points previously mentioned. Note that $\partial K$ is invariant by $f$ and by $f^{-1}$, and hence so is $\mathring{K}$. By Equation (\ref{equidistribution}), since the support of $f^{n*}\delta_x$ is included in $\mathring{K}$ for all $n$, every open subset of $\mathbb{C}$ intersecting $J=\partial K$ also intersects $\mathring{K}$, thus $\overline{\mathring{K}}=K$. 
\qed
\end{proof}

\section{Study of the obstruction to the {\L}S condition}

For $n\geq 1$, let \begin{equation}\label{Oc}\displaystyle O_c:=\{z\in\mathbb{C}^n: dist(z,A)< 1, G_A(z)< c \cdot dist(z,A)^{1/c}\}.\end{equation}
Note that the sequence of open sets $O_c$ is increasing with $c$ for $c<1$. The {\L}S condition is satisfied by a compact nonpolar set $A\subset\mathbb{C}^n$, L-regular and polynomially convex, if and only if the set $$\displaystyle I:=\bigcap_{c>0} \overline{O_c}\subset \partial A$$ is empty. We call $I$ {\it{the set of obstruction points to the {\L}S  condition}}.

\begin{exemple}[\cite{BK}] Counter-example : If $A$ is the union of two disks of radius $1$, tangent to each other at the origin, then it does not satisfy the {\L}S condition; the set of obstruction points is  $I=\{0\}\neq \emptyset$. 
\end{exemple}

\noindent The following result provides more insight into the structure of the complementary of $O_c$. We prove it for $n\geq 1$.

\begin{pro}\label{lemme1}
Let $A\subset\mathbb{C}^n$, $n\geq 1$, be a nonpolar, L-regular and polynomially convex compact set. Suppose that the pluricomplex Green function $G_A$ is pluriharmonic outside of A (harmonic if $n=1$).\\
 Then, there exists $c_0>0$ such that $\forall c\in \ ]0,c_0]$, $\partial A$ is included in the boundary of the open set $\{z \in \mathbb{C}^n: G_A(z)>c \cdot dist(z,A)^{1/c}\}$. 
\end{pro}

\begin{proof}
Let $\mu$ denote the positive measure $\dis \frac{i}{\pi}\partial \overline{\partial} G_A\wedge \omega^{n-1}$ on $\mathbb{C}^n$, where $$\dis \omega := \frac{i}{2\pi}\partial \overline{\partial} \log (1+\|z\|^2)$$ is the Fubini-Study form. Note that the support of the measure $\mu$ is exactly $\partial A$. Indeed, $ supp(\mu)\subset\partial A$ since $\frac{i}{\pi}\partial \overline{\partial} G_A=0$ in $\mathbb{C}^n\setminus \partial A$ by hypothesis. On the other hand, if there existed $x\in \partial A\setminus supp(\mu)$, then $G_A$ would be (pluri)harmonic in a neighborhood of $x$, hence null in this neighborhood, which can not happen because $A$ is polynomially convex.

\noindent Let us suppose by contradiction that $\forall c_0>0,\text{ }\exists c\in ]0,c_0],\text{ }\exists x \in \partial A,\text{ }\exists r>0,\text{ }$
$B(x,r)\cap \{z \in \mathbb{C}^n, G_A(z)>c \cdot dist(z,A)^{1/c}\}=\emptyset.$

\noindent Thus we can take $\dis c'\in \left]0,\frac{1}{4n}\right[$, $x'\in\partial A$, and $r'>0$, such that $$G_A(z)\leq c'\cdot dist(z,A)^{\frac{1}{c'}}, \ \forall  z \in B(x',r').$$

\noindent Denote $r_0:=\frac{r'}{2}$. Then, $\forall r<r_0$, $\forall x\in B(x',r_0)\cap \partial A$, the Chern-Levine-Nirenberg inequality implies : $$\displaystyle \mu\big(B(x,r)\big) \leq k\cdot r^{-2n}\sup_{B(x,2r)} G_A \leq c'\cdot k\cdot (2r)^{\frac{1}{c'}-2n},$$
for some constant $k>0$ independent of $r$, $r_0$, $x'$ and $c'$.

With the notation $\displaystyle \nu:=\frac{\mu}{\mu\big(B(x',r_0)\big)}\mathbf{1}_{B(x',r_0)}$, where $\displaystyle\mathbf{1}_{B(x',r_0)}$ is the characteristic function of $\displaystyle B(x',r_0)$, the measure $\nu$ is a probability measure, and we can rewrite the previous inequality : $\forall r>0$, $\forall x\in B(x',r_0)\cap \partial A$, $$\displaystyle \nu\big(B(x,r)\big)\leq \frac{ c'\cdot k}{{\mu\big(B(x',r_0)\big)}}\cdot (2r)^{\frac{1}{c'}-2n}.$$

Then, by Frostman Lemma (see for example Lemma 10.2.1 in \cite{Be}), the Hausdorff dimension of $\partial A\cap B(x_0,r_0)$ is strictly greater than $2n$ for our choice $c'<\frac{1}{4n}$, which gives a contradiction. (Recall that Frostman Lemma ensures that, if $m$ is a probability measure on a metric space $E$ verifying $m\big(B(x,r)\big)<q \cdot r^{\alpha}$ for all $x\in E$, $r>0$, with fixed $q>0$, $\alpha >0$, then the Hausdorff dimension of $E$ is greater than $\alpha$).

We thus conclude that $\exists c_0>0$, $\forall c\in ]0,c_0]$, $\forall x\in\partial A$, $\forall r>0$: $$\displaystyle \dis B(x,r)\cap \{z \in \mathbb{C}^n, G_A(z)>c \cdot dist(z,A)^{1/c}\}\neq \emptyset,$$
which proves the statement. 
\qed
\end{proof}

\section{Proof of the main theorem}

We will need the following result of Poletsky (Corollary p. 170 in \cite{Po}, see also \cite{Po2}), generalized by Rosay (\cite{Ro}). Let $U$ a connected complex manifold of dimension $n\geq 1$. We denote by $\mathcal{H}_{z,U}$ the set of holomorphic functions $h:V_h\rightarrow U$ from a neighbourhood $V_h$ of $\overline{\Delta}=\{|z|\leq 1\}\subset\mathbb{C}$ (possibly depending on $h$) into $U$ such that $h(0)=z$. We also denote by $PSH(U)$ the set of plurisubharmonic functions defined on $U$. 

\begin{pro}\label{poletsky} Let $u:U\rightarrow \mathbb{R}$ be an upper semi-continuous function. With the previous notations, the function defined by $$\displaystyle \tilde{u}(z):=\frac{1}{2\pi}\inf_{f\in \mathcal{H}_{z,U}}\int_0^{2\pi} u(f(e^{i\theta}))d\theta ,$$ if it is not everywhere equal to  $-\infty$, belongs to $PSH(U)$ and verifies $\tilde{u}\leq {u}$. Moreover, this function $\tilde{u}$ is maximal among all the functions in $PSH(U)$ verifying this inequality.
\end{pro}

\begin{rmq}We deduce from Proposition \ref{poletsky} the following property of antisubharmonic functions, i.e. functions with subharmonic opposite. Let $B:=B(a,r)\subset\mathbb{C}$ be an open ball, $\dis u:\overline{B}\rightarrow\mathbb{R}$ a continuous function, antisubharmonic in $B$. Then $\hat{u}:B\rightarrow\mathbb{R}$ is an harmonic function, with the same boundary values as $u$, in the sense that $\displaystyle \lim_{z\rightarrow z_0}\hat{u}=u(z_0)$ for $z_0\in \partial B$. \\
Indeed, given a continuous function $g:\overline{B}\rightarrow \mathbb{R}$, denote by $\tilde{g}:B\rightarrow \mathbb{R}$ the solution of the Dirichlet problem in $B$ with boundary condition $g_{|_{\partial B}}$, that is to say, the unique continuous function defined on $\overline{B}$ which is harmonic in $B$ and equal to $g$ on $\partial{B}$. Then $v:=\max(\tilde{u},\hat{u})$ is a subharmonic function with the same values as $u$ on $\partial B$. Since $u$ is antisubharmonic, we have $\tilde{u}\leq u$. Thus $$\hat{u}\leq v\leq u.$$ Since $\hat{u}$ is maximal among the subharmonic functions which are $\leq u$ in $B$ and equal to $u$ on $\partial B$, we conclude that $\hat{u}=v$, and hence $\tilde{u}=\hat{u}$. \\
Thanks to Theorem 3.1.4 in \cite{K}, the conclusion is the same if $B$ is a ball in $\mathbb{C}^n$, when substituting the expression "harmonic function" by "maximal plurisubharmonic function", and the expression "antisubharmonic function" by "antiplurisubharmonic function".\end{rmq}

Let $U\subset \mathbb{C}^n$, $n\geq 1$, be a bounded open set. Denote by $\lambda$ the normalized Lebesgue measure on the unit circle $\partial\mathbb{U}\subset\mathbb{C}$. Denote also by $\Lambda_{z,U}$ the set of measures of the form $h_*\lambda(\cdot):=\lambda(h^{-1}(\cdot))$, where $h:V_h\rightarrow U$ is an holomorphic function defined in a neigborhood $V_h$ (possibly depending on $h$) of the closed unit disk $\overline{\mathbb{U}}$, such that $h(0)=z$. Note that that the Dirac measure $\delta_z$ belongs to $\Lambda_{z,U}$. An immediate consequence of Proposition \ref{poletsky} is the following corollary, where $\mathbf{1}_G$ denotes the characteristic function of $G\subset\mathbb{C}^n$ :

\begin{cor}\label{klimek}Let $U\subset \mathbb{C}^n$ be a bounded open set, and $A\subset U$ a $L$-regular nonpolar compact set satisfying $\overline{\mathring{A}}=A$. Then $$\dis \frac{1}{2\pi}\inf_{f\in \mathcal{H}_{z,U}}\int_0^{2\pi} -\mathbf{1}_{A}\circ f(e^{i\theta})d\theta=-\sup_{\mu_z\in\Lambda_{z,U}}\mu_z(A)   =G_{A,U}(z).$$\end{cor}

\noindent Recall that we denote by $K$ the filled Julia set of a polynomial application $f:\mathbb{C}\rightarrow\mathbb{C}$ of degree $\geq 2$, and $dist(\cdot,\cdot)$ the euclidean distance on $\mathbb{C}^n$. Let us prove the main result stated in the introduction :

\begin{thm}\label{thm} Let $K\subset \mathbb{C}$ be the filled Julia set of a polynomial $f:\mathbb{C}\rightarrow \mathbb{C}$ of degree $d\geq 2$, of non-empty interior. Then $K$ satisfies the {\L}S condition.\end{thm}

\begin{proof}
For $b\in \mathbb{R}^{+}\setminus \{0\}$, denote $U_b:=\{G_K<b\}\subset \mathbb{C}$. For $l\in \mathbb{R}^{+}\setminus \{0\}$, denote also $K_l:=\{z\in\mathbb{C}\text{ }|\text{ }dist(z,K)\leq l\}$. Then choose $a>0$ such that $\dis K_{2}\subset f^{-1}(U_a)$. Note that $f^{-1}(U_a)=U_{\frac{a}{d}}\subset\subset U_a$ by (\ref{invariance}). Denote by $\mathcal{C}_a$ the corona $U_a\setminus f^{-1}(U_a)$. There exists $\delta\in ]0,1[$ such that \begin{equation}\label{delta}\dis G_{K_{2, U_a}}+1\geq\delta (G_{K,U_a}+1)\text{ on }\mathcal{C}_a.\end{equation} Take $\dis c\in \left]0,\frac{\delta}{2a}\right[$, sufficiently small to have $\overline{O_c}\subset f^{-1}(U_a)$ and $\left(\frac{1}{c^2}\right)^c<2$.
We have $\forall \epsilon\in ]0,2]$, $\forall y\in U_a$,
\begin{align*}
\dis c\cdot dist(y,K)^{\frac{1}{c}}&\geq \inf_{\mu_y\in \Lambda_{y,U_a}}\int_{U_a} c\cdot dist(\cdot,K)^{\frac{1}{c}}d\mu_y\\
&\geq \inf_{\mu_y\in \Lambda_{y,U_a}}\int_{U_a\setminus K_{\epsilon}} c\cdot dist(\cdot,K)^{\frac{1}{c}}d\mu_y \\
&\geq \left(\min_{\overline{U_a}\setminus \mathring{K_{\epsilon}}}c\cdot dist(\cdot, K)^{\frac{1}{c}}\right)\inf_{\mu_y\in \Lambda_{y,U_a}}\int_{U_a\setminus K_{\epsilon}}d\mu_y \\
&= c\epsilon^{\frac{1}{c}}(G_{K_\epsilon,U_a}+1)(y).
\end{align*}
The first inequality comes from the fact that the Dirac measure $\delta_y$ belongs to $\Lambda_{y,U_a}$. The last inequality comes from Corollary \ref{klimek}, whose application is allowed by Lemma \ref{julia}. Then taking $\epsilon=\left(\frac{1}{c^2}\right)^c<2$, we obtain in $U_a$:  \begin{equation}\label{estimation}\dis c\cdot dist(\cdot,K)^{\frac{1}{c}}\geq \frac{1}{c}(G_{K_\epsilon,U_a}+1).\end{equation}

Now suppose, by contradiction, that $O_c\neq \emptyset$ (see Equation (\ref{Oc}) for definition). Thanks to the fact that  $c<\frac{\delta}{2a}$, we can choose $x\in O_c\setminus  \{G_{K}<\frac{2ac^2}{\delta}dist(\cdot,K)^{\frac{1}{c}}\}$. 

Note that $x\in O_c$ implies a "slow growth" of $(f^n(x))_n$, in the sense that $\forall n\geq 1$ such that $f^n(x)\notin O_c$, we have 
$$\dis\frac{1}{d^n}c\cdot dist(f^n(x),K)^{\frac{1}{c}}\leq G_K(x)< c\cdot dist(x,K)^{\frac{1}{c}},$$ and hence

\begin{equation}\label{slow}\dis dist(f^n(x),K)< d^{nc} dist(x,K).\end{equation}

\noindent Since $\dis U_a\setminus K = \bigcup_{i\geq 0}f^{-i}(\mathcal{C}_a)$ by (\ref{invariance}), there exists $N>0$ such that $f^N(x)\in \mathcal{C}_a$. Equations (\ref{slow}), (\ref{estimation}), (\ref{delta}), (\ref{relation}), then (\ref{invariance}), give 
\begin{align*}
\dis c\cdot  dist(x,K)^{\frac{1}{c}} &\geq \frac{c}{d^N} dist\left(f^N(x),K\right)^{\frac{1}{c}}\\
&\geq \frac{1}{cd^N}(G_{K_{\epsilon},U_a}+1) \circ f^N(x) &  \\
\\ &\geq \frac{\delta}{cd^N}(G_{K,U_a}+1)\circ f^N(x)\\
&= \frac{\delta}{ca}G_K(x).
\end{align*}
But this contradicts our assumption $x\notin \{G_{K}<\frac{2ac^2}{\delta}dist(x,K)^{\frac{1}{c}}\}$. We conclude that $O_c=\emptyset$. In other words, $K$ satisfies the {\L}S condition. 
\qed
\end{proof}

\begin{rmq}We note that if $f$ is assumed to be {\bf{hyperbolic}}, that is to say if $f$ do not have critical points in $J$, there exist a constant $b>0$ and a neighborhood of $K$ in which \begin{equation}\label{control}\dis dist\left(f(\cdot), K\right)\geq b\cdot dist(\cdot, K).\end{equation}Indeed, it is sufficient to etablish this inequality outside $K$. Let then $V$ be a neigborhood of $K$ in which $|f'|\geq a$ for some $a>0$, let $z\in V\setminus K$, and $z_0\in J$ such that $f(z_0)\in J$ achieves the distance $dist(f(z),J)$. Then Theorem 1 of \cite{D} shows the existence of a constant $k>0$ (depending only on the degree of $f$) and of a point $z_1\in J=\partial K$, such that $$\dis dist(f(z),K)=dist\left(f(z),f(z_0)\right)\geq a\cdot k \cdot dist(z,z_1)\geq a\cdot k \cdot dist(z,K).$$
\noindent In the particular case where $b\geq 1$ in (\ref{control}), we obtain a simpler proof of Theorem \ref{thm}, and a more quantitative estimation for $c$ in Equation (\ref{Oc}). Indeed, suppose $O_c\neq \emptyset$ with $O_c\subset \subset V$. We can choose $x\in O_c$ such that $f(x)\notin O_c$. Then, (\ref{slow}) together with (\ref{control}) give $$\dis c>\frac{\log{b}}{\log d}.$$
\end{rmq}

\begin{acknowledgements}
We thank Laurent Gendre and Ta\"ib Belghiti for their reading and suggestions, as well as Marta Kosek and the whole organization committee of the conference "On Constructive Theory of Functions" for their invitation to expose a preliminary version of this article in Poland.
\end{acknowledgements}

\end{document}